\documentclass[12pt,reqno]{amsart}

\usepackage{amsfonts}
\usepackage{amscd}
\usepackage{amssymb}
\usepackage{amsfonts}
\usepackage{amscd}
 \usepackage{amsmath} 
\usepackage{mathrsfs}
\usepackage{enumerate}
\usepackage{float} 
\usepackage[pdftex]{graphicx}
\allowdisplaybreaks

\usepackage{hyperref}

\date{\today}

\newtheorem{thm}{Theorem}[section]

\newtheorem{lem}[thm]{Lemma}

\newtheorem{prop}[thm]{Proposition}
\theoremstyle{definition}

\theoremstyle{remark}
\newtheorem{rem}[thm]{Remark}

\numberwithin{equation}{section}

\newcommand{\R}{\mathbb R}

\newcommand{\He}{\mathbb H}

\newcommand{\C}{{\mathbb C}}

\renewcommand{\Im}{\operatorname{Im}}
 
\usepackage{color}
\newcommand{\blue}[1]{\textcolor{blue}{#1}}
\newcommand{\red}[1]{\textcolor{red}{#1}}

\title[Uncertainty principles on  $\He^n$]
{Analogues of theorems of Chernoff and Ingham\\ 
on the Heisenberg group }

\author[ Ganguly and Thangavelu]{ Pritam Ganguly and Sundaram Thangavelu}

\address[P. Ganguly, S. Thangavelu]{Department of Mathematics\\
Indian Institute of Science\\
560 012 Bangalore, India}

\email{pritamg@iisc.ac.in, veluma@iisc.ac.in}

 \date{}
 
\begin{document}

\maketitle

\begin{abstract} We prove an analogue of Chernoff's theorem for the Laplacian $ \Delta_{\He} $ on the Heisenberg group $ \He^n.$ As an application, we prove  Ingham type theorems for the group Fourier transform on $ \He^n $  and also for the spectral projections associated to the sublaplacian. 
	   
 \end{abstract}

\section{Introduction}
Uncertainty principles in harmonic analysis have thrilled mathematicians for a long time. One of the several avatars  of the uncertainty principle,  dealing with the best possible decay admissible for the Fourier transform of a nontrivial function which vanishes on an open set was studied by Ingham in 1934.  Proving analogues of this result in various settings has  received considerable attention in recent years.  In some of the works, a theorem of Chernoff  on quasi analytic functions  has played an important role in proving Ingham type theorems. In this paper our aim is two-fold. We first prove an analogue of Chernoff's theorem for the full Laplacian  $\Delta_{\mathbb{H}} $ on the Heisenberg group $ \He^n$ and then use it prove  Ingham type theorems for the (operator valued) Fourier transform  on $\mathbb{H}^n$ and also for the spectral projections associated to the sublaplacian $ \mathcal{L}.$

 Chernoff's theorem on $ \R^n $ is to be viewed as a higher dimensional analogue of Denjoy-Carleman theorem which characterizes quasi analytic functions. In 1950, instead of using partial derivatives, Bochner used iterates of Laplacian $ \Delta $ to study quasi analytic functions on $ \R^n.$ Later in 1972, by using operator theoretic arguments, Chernoff \cite{C} improved the result of Bochner and proved the following result.

\begin{thm}\cite[Chernoff]{C}
		Let $f$ be a smooth function on $\mathbb{R}^n.$ Assume that $\Delta^mf\in L^2(\mathbb{R}^n)$ for all $m\in \mathbb{N}$ and $\sum_{m=1}^{\infty}\|\Delta^mf\|_2^{-\frac{1}{2m}}=\infty.$ If $f$ and all its partial derivatives   vanish at  a point  $ a \in \R^n$, then $f$ is identically zero.
\end{thm}

As Chernoff's theorem is a useful tool in establishing uncertainty principles of Ingham's type, proving analogues of  Theorem 1.1 in contexts other than Euclidean spaces have received considerable attention in recent years.  For noncompact Riemannian symmetric spaces $ X = G/K$,  without any restriction on the rank, the following weaker version of Theorem 1.2 has been proved in Bhowmik-Pusti-Ray \cite{BPR}.

\begin{thm}[Bhowmik-Pusti-Ray]
	\label{thm-BPR}
 Let ${X}=G/K$ be a  noncompact Riemannian symmetric space and let $ \Delta_X $ be the associated Laplace-Beltrami operator. Suppose $f\in C^{\infty}(X)$ satisfies $\Delta_X^mf\in L^2(X)$ for all $m\geq 0$ and $\sum_{m=1}^{\infty}\|\Delta_{X}^mf\|_2^{-\frac{1}{2m}}=\infty.$ If $ f $ vanishes on a non empty open set,  then $f$ is identically zero.
\end{thm} 

Observe that in the above result, the function $ f $ is assumed to vanish on an open set.  Proving an exact analogue of Chernoff's theorem is still open though there are some partial results. Recently in \cite{BPR1} the authors have proved an exact analogue of Chernoff's theorem for $K$-biinvariant functions on the group $ G.$   Under the assumption that $ X $ is of rank one, we have proved an exact analogue of Chernoff's theorem in a joint work with R. Manna \cite{GMT}:

\begin{thm}[Ganguly-Manna-Thangavelu]
	\label{C}
 Let ${X}=G/K$ be a rank one Riemannian symmetric space of noncompact type. Suppose $f\in C^{\infty}(X)$ satisfies $\Delta_X^mf\in L^2(X)$ for all $m\geq 0$ and $\sum_{m=1}^{\infty}\|\Delta_{X}^mf\|_2^{-\frac{1}{2m}}=\infty.$ If   ${H}^lf(eK)=0$ for all $l\geq 0$ then $f$ is identically zero.
\end{thm} 

In the above, $ H $ is any nonzero element of the one dimensional Lie algebra $ \mathfrak{a} $ occurring in the Iwasawa decomposition $ \mathfrak{g} = \mathfrak{k} \oplus \mathfrak{a} \oplus \mathfrak{n}.$

In view of the above results, it is an interesting problem to study Chernoff's theorem for the sublaplacian $ \mathcal{L}$ on the Heisenberg group $ \He^n.$   The following version  of Chernoff's theorem  has been proved in \cite{BGST}.

\begin{thm} [Bagchi-Ganguly-Sarkar-Thangavelu]
Let $f$ be a smooth function on $\mathbb{H}^n$  satisfying  $ f(z,t) = f_0( |(z,t)|) $ where $ |(z,t)| = (|z|^4+ t^2)^{1/4} $ is the Koranyi norm on $ \He^n.$ 
 Assume that $\mathcal{L}^mf\in L^2(\mathbb{H}^n)$ for all $m\in \mathbb{N}$ and $\sum_{m=1}^{\infty}\|\mathcal{L}^mf\|_2^{-\frac{1}{2m}}=\infty$. If $f$ and  all its partial derivatives vanish at $0$, then $f$ is identically zero.
\end{thm}

Observe that as in the case of  symmetric spaces of arbitrary rank studied in \cite{BPR1} we have also imposed an extra condition on $ f.$  It is still an open problem to prove the above result without the  extra assumption on $ f.$
However, in this paper we consider the full Laplacian $ \Delta_{\He}$ instead of $ \mathcal{L} $ and  prove  the following version which is the analogue of Theorem 1.2 in our context.

\begin{thm}\label{cher} 
Let $f$ be a smooth function on $\mathbb{H}^n$ such that  $\Delta_{\He}^mf\in L^2(\mathbb{H}^n)$ for all $m\in \mathbb{N}$ and $\sum_{m=1}^{\infty}\|\Delta_{\He}^mf\|_2^{-\frac{1}{2m}}=\infty$. If $f$ vanishes on a nonempty open set, then $f$ is identically zero.
\end{thm}

As an application of this result, we are able to strengthen the Ingham's theorem proved in \cite{BGST}.  In Theorem 1.3 in \cite{BGST} we have investigated the  admissible decay of the Fourier transform $ \hat{f}(\lambda) $ of a nontrivial function $ f $ on $ \He^n.$ As $ \hat{f}(\lambda) $ is operator valued, the decay is measured in terms of the Hermite operator $ H(\lambda) $ in the following form:
\begin{equation}\label{Ing-decay} 
 \hat{f}(\lambda)^\ast \hat{f}(\lambda) \leq C e^{- 2 \sqrt{H(\lambda)} \,\Theta(\sqrt{H(\lambda))}} 
 \end{equation}
for a non-negative function $ \Theta $ defined  on $ [0,\infty).$ More precisely, the following theorem has been proved.

\begin{thm}\label{ingh-hei}
Let $ \Theta(\lambda) $ be a nonnegative  function on $ [0,\infty)$ such that $ \Theta(\lambda) $ decreases to zero when $ \lambda \rightarrow \infty $  and satisfies the condition $ \int_1^\infty \Theta(t) t^{-1} dt <\infty.$  Then there exists a nonzero compactly supported continuous function $ f $ on $ \He^n$ whose  Fourier transform $ \widehat{f} $ satisfies  the estimate (\ref{Ing-decay})
 Conversely, for any nontrivial integrable function $f $ vanishing on a neighbourhood of zero satisfying  the extra assumption   $ f(z,t) = f_0( |(z,t)|), $  the estimate (\ref{Ing-decay}) cannot hold unless  $ \int_1^\infty \Theta(t) t^{-1} dt <\infty .$ 
\end{thm}

 In this paper we show that the extra condition  on $ f $ can be dispensed with if we slightly strengthen the condition (\ref{Ing-decay}).  Let $ g $ be a function on $ \R $ whose Euclidean Fourier transform satisfies the estimate
$ |\hat{g}(\lambda)| \leq C e^{-|\lambda|\,\Theta(|\lambda|)} $ for all $ \lambda \in \R.$  If $ f_0 $ satisfies (\ref{Ing-decay}), then the function $ f(z,t) = \int_{-\infty}^\infty f_0(z,t-s)g(s) ds $ satisfies the condition
\begin{equation}\label{Ing-decay-mod}
 \hat{f}(\lambda)^\ast \hat{f}(\lambda) \leq C \, e^{-2|\lambda|\,\Theta(|\lambda|)}e^{-2 \sqrt{H(\lambda)} \,\Theta(\sqrt{H(\lambda))}} .
\end{equation}
By combining the first part of Theorem \ref{ingh-hei} and the classical theorem of Ingham it is not difficult to prove the following result.
\begin{thm}\label{ingh-easy} Let $ \Theta(\lambda) $ be a nonnegative  function on $ [0,\infty)$ such that $ \Theta(\lambda) $ decreases to zero when $ \lambda \rightarrow \infty .$ Then there exists a nonzero compactly supported continuous function $ f $ on $ \He^n$ whose  Fourier transform $ \widehat{f}(\lambda) $ satisfies  the estimate (\ref{Ing-decay-mod}) if and only if $ \int_1^\infty \Theta(t) t^{-1} dt <\infty.$
\end{thm}

The proof of this theorem which will be presented in Section 4 is not difficult. Thus, if $ \int_1^\infty \Theta(t) t^{-1} dt = \infty,$ then we cannot have any non trivial function $ f $ with compact support whose Fourier transform satisfies (\ref{Ing-decay-mod}). However, if we only assume that $ f $ vanishes on a nonempty open set, then proving that $ f $ is identically zero is much more difficult. We need to make use of the full power of Theorem \ref{cher}.  In this paper we prove the following result.

\begin{thm}\label{ingh-hei-1}
Let $ \Theta(\lambda) $ be  a nonnegative  function on $ [0,\infty)$ such that it  decreases to zero when $ \lambda \rightarrow \infty $  and satisfies the conditions $  \int_1^\infty \Theta(t) t^{-1} dt = \infty.$  Let $ f $ be  an integrable  function on $ \He^n $ whose Fourier transform  satisfies  the estimate (\ref{Ing-decay-mod}). Then $ f $ cannot vanish  on any  nonempty open set unless it is identically zero.
\end{thm}

Actually, we  prove a refined version of the above theorem by replacing (\ref{Ing-decay-mod}) by a decay assumption on the spectral projections associated to the sublaplacian. In order to motivate our result, it is instructive to recast the condition (\ref{Ing-decay-mod}) in terms of a different but equivalent definition of Fourier transform. In the above, $ \hat{f}(\lambda), \lambda \in \R^\ast $ is defined in terms of the Schr\"odinger representation $ \pi_\lambda $ of  $ \He^n $ realised on the Hilbert space $ L^2(\R^n).$ Instead, we can consider functions $f $ on $ \He^n $ as right $ U(n) $ invariant functions on the Heisenberg motion group $ G_n = \He^n \ltimes U(n) $ which allows us to consider $ \rho_k^\lambda(f) $ for a family of class-1 representations of $ G_n$ indexed by $ \lambda \in \R^\ast $ and $ k \in \mathbb{N}$ realised on certain Hilbert spaces $ \mathcal{H}_k^\lambda$ which are some explicit function spaces on $ \He^n.$

The representations $ \rho_k^\lambda $ when restricted to $ \He^n $ are not irreducible but split into finitely many irreducible unitary representations each one being equivalent to $ \pi_\lambda.$ Since $ \rho_k^\lambda$ are class-1 representations of $ G_n $ each of them has a unique $U(n)$-fixed vector in $ \mathcal{H}_k^\lambda $ which we denote by $ e^{n-1}_{k,\lambda}(z,t).$ Thus the scalar valued function $ f \rightarrow  \rho_k^\lambda(f)e^{n-1}_{k,\lambda}(z,t) =  e^{i\lambda t}\hat{f}(\lambda,k, z) $ can be considered as the analogue of the Helgason Fourier transform on Riemannian symmetric spaces of noncompact type. 
It can be shown that $ e^{i\lambda t} \hat{f}(\lambda,k, z) $ are eigenfunctions of the sublaplacian  with eigenvalues $ (2k+n)|\lambda| $ and $ f $ can be recovered  by the formula
\begin{equation}
f(z,t) = (2\pi)^{-n-1}\, \int_{-\infty}^\infty   e^{i\lambda t}\,  \big(\sum_{k=0}^\infty  \rho_k^\lambda(f)e^{n-1}_{k,\lambda}(z,0) \big) |\lambda|^n d\lambda.
\end{equation}
 We can thus view the above as the spectral decomposition of the sublaplacian. 
 Moreover, 
 \begin{equation}  \frac{(k+n-1)}{k! (n-1)!} \| \rho_k^\lambda(f)\|_{HS}^2  = (2\pi)^{-n} |\lambda|^n \,  \int_{\C^n} | \rho_k^\lambda(f)e^{n-1}_{k,\lambda}(z,0)|^2 dz . 
 \end{equation}
 It is not difficult to check that the  condition (\ref{Ing-decay-mod})  leads to the estimate
 \begin{equation}
   \frac{(k+n-1)}{k! (n-1)!} \| \rho_k^\lambda(f)\|_{HS}^2  \leq C\, e^{-2|\lambda|\,\Theta(|\lambda|)} e^{- 2  \sqrt{(2k+n)|\lambda|}\, \Theta(\sqrt{(2k+n)|\lambda|})} 
   \end{equation}
 and it turns out that Theorem \ref{ingh-hei-1} can be proved solely under the above condition on $ \| \rho_k^\lambda(f)\|_{HS},$ (see Subsection 4.1).
 
 However we can do better than this: instead of assuming decay estimates on $ \| \rho_k^\lambda(f)\|_{HS}$ we can impose pointwise estimates on the spectral projections $ \rho_k^\lambda(f)e^{n-1}_{k,\lambda}(z,t) $ and  prove the following version of Ingham's theorem.
 
 \begin{thm}\label{ingh-hei-2}
Let $ \Theta(\lambda) $ be a non-negative  function on $ [0,\infty)$ such that $ \Theta(\lambda) $ decreases to zero when $ \lambda \rightarrow \infty $  and satisfies the condition $ \int_1^\infty \Theta(t) t^{-1} dt = \infty.$  Let $ f $ be  a  nontrivial integrable function on $ \He^n $   which vanishes on an open set $ V .$ Then its  (Helgason) Fourier transform cannot satisfy the uniform estimate 
$$  \sup_{(z,t) \in V}  |\rho_k^\lambda(f)e^{n-1}_{k,\lambda}(z,t)| \leq C\, e^{-|\lambda|\,\Theta(|\lambda|)} e^{-\sqrt{(2k+n)|\lambda|}\, \Theta(\sqrt{(2k+n)|\lambda|})} .$$
\end{thm}

As we have already mentioned, $\rho_k^\lambda(f)e^{n-1}_{k,\lambda}(z,t) $ are eigenfunctions of $ \mathcal{L} $ and hence the above theorem is a version of Ingham's theorem for the spectral projections.  Earlier we have proved such theorems for spectral projections associated to certain elliptic differential operators, see \cite{GT2} for $ \Delta $ on non compact Riemannian symmetric spaces and \cite{GT1} for the Hermite and special Hermite operators and $ \Delta $ on compact symmetric spaces.

 The plan of the paper is as follows. In the next section we collect necessary preliminaries on the Heisenberg group, Heisenberg motion group and Laguerre expansions. In section 3, we prove an analogue of Chernoff's theorem for the generalised Laplacian and  we use this to prove an analogue of Chernoff's theorem for the full Laplacian on the Heisenberg group. Finally in section 4, we prove Ingham type uncertainty principles for the group Fourier transform and spectral projections associated to the sublaplacian.
\section{Preliminaries on the Heisenberg group}
We develop the required background for the Heisenberg group.  General references for this section are the monographs of Thangavelu \cite{TH1},\cite{TH2} and \cite{TH3}. Also see the book \cite{F} of Folland.
\subsection{Fourier transform on the Heisenberg group}	Let $\mathbb{H}^n:=\mathbb{C}^n\times\mathbb{R}$ be the $(2n+1)$- dimensional Heisenberg group with the group law
	 
	$$(z, t).(w, s):=\big(z+w, t+s+\frac{1}{2}\Im(z.\bar{w})\big),\ \forall (z,t),(w,s)\in \mathbb{H}^n.$$ This is a step two nilpotent Lie group where the Lebesgue measure $dzdt$ on $\mathbb{C}^n\times\mathbb{R}$ serves as the Haar measure. The representation theory of $\mathbb{H}^n$ is well-studied in the literature. In order to define Fourier transform, we use the Schr\"odinger representations as described below.  
	
	For each non zero real number $ \lambda $ we have an infinite dimensional representation $ \pi_\lambda $ realised on the Hilbert space $ L^2( \R^n).$ These are explicitly given by
	$$ \pi_\lambda(z,t) \varphi(\xi) = e^{i\lambda t} e^{i\lambda(x \cdot \xi+ \frac{1}{2}x \cdot y)}\varphi(\xi+y),\,\,\,$$
	where $ z = x+iy $ and $ \varphi \in L^2(\R^n).$ These representations are known to be  unitary and irreducible. Moreover, by a theorem of Stone and Von-Neumann, (see e.g., \cite{F})  upto unitary equivalence these account for all the infinite dimensional irreducible unitary representations of $ \mathbb{H}^n $ which act as $e^{i\lambda t}I$ on the center. Also there is another class of finite dimensional irreducible representations. As they  do not contribute to the Plancherel measure  we will not describe them here.
	
	The Fourier transform of a function $ f \in L^1(\mathbb{H}^n) $ is the operator valued function  defined on the set of all nonzero reals, $ \R^\ast $ given by
	$$ \hat{f}(\lambda) = \int_{\mathbb{H}^n} f(z,t) \pi_\lambda(z,t)  dz dt .$$  Note that $ \hat{f}(\lambda) $ is a bounded linear operator on $ L^2(\R^n).$ It is known that when $ f \in L^1 \cap L^2(\mathbb{H}^n) $ its Fourier transform  is actually a Hilbert-Schmidt operator and one has
	$$ \int_{\mathbb{H}^n} |f(z,t)|^2 dz dt = (2\pi)^{-(n+1)}\int_{-\infty}^\infty \|\widehat{f}(\lambda)\|_{HS}^2 |\lambda|^n d\lambda  $$
	where $\|.\|_{HS}$ denote the Hilbert-Schmidt norm. 
	The above allows us to extend  the Fourier transform as a unitary operator between $ L^2(\mathbb{H}^n) $ and the Hilbert space of Hilbert-Schmidt operator valued functions  on $ \R $ which are square integrable with respect to the Plancherel measure  $ d\mu(\lambda) = (2\pi)^{-n-1} |\lambda|^n d\lambda.$ We polarize the above identity to obtain 
	$$\int_{\He^n}f(z,t)\overline{g(z,t)}dzdt=\int_{-\infty}^{\infty}tr(\widehat{f}(\lambda)\widehat{g}(\lambda)^*)~d\mu(\lambda).$$ Also for suitable function $f$ on $\He^n$ we have the following inversion formula
	$$f(z,t)=\int_{-\infty}^{\infty}tr(\pi_{\lambda}(z,t)^*\widehat{f}(\lambda))d\mu(\lambda).$$
	Now by definition of $\pi_{\lambda}$ and $\hat{f}(\lambda)$ it is easy to see that 
	$$\widehat{f}(\lambda)=\int_{\C^n}f^{\lambda}(z)\pi_{\lambda}(z,0)dz $$ where 
	$f^{\lambda}$ stands for the inverse Fourier transform of $f$ in the central variable:
	$$f^{\lambda}(z):=\int_{-\infty}^{\infty}e^{i\lambda.t}f(z,t)dt.$$

		 Given a  suitable function $g$ on $\C^n$, we consider the following   operator valued function  defined by
		$$ W_{\lambda}(g):=\int_{\C^n}g(z)\pi_{\lambda}(z,0)dz.
		$$ With these notations we note that  $\hat{f}(\lambda)=W_{\lambda}(f^{\lambda}).$  These transforms are called the Weyl transforms. We have the following Plancherel formula for Weyl transform  (See \cite[2.2.9, Page no-49]{TH3})
		\begin{equation}
			\label{wplan}
			 \|W_{\lambda}(g)\|^2_{HS}|\lambda|^n=(2\pi)^n\|g\|_2^2, ~g\in L^2(\mathbb{C}^n).
		\end{equation}   

Now we move our attention to spherical means on $\He^n$ which was introduced by Nevo-Thangavelu in \cite{NeT}. This will play a very important role in proving Chernoff's theorem for the full Laplacian.  

\subsection{Spherical means on $\mathbb{H}^n$} We  consider the spherical means of  a function $f$ on $\mathbb{H}^n$ defined by 
\begin{equation}
	\label{eq:defin}
	f\ast \mu_r(z,t)=\int_{|w|=r}f\Big(z-w,t-\frac12\Im z\cdot \overline{w}\Big)\,d\mu_r(w)
\end{equation}
where $\mu_r$ is the normalised surface measure on the sphere $S_r=\{(z,0):|z|=r\}$ in $\mathbb{H}^n$. In the following, we describe the special Hermite expansion of the spherical means which will play a very important role later. In order to do that, we consider the Laguerre function of type $(n-1)$ defined by
$$\varphi_k^{n-1}(r):=L_k^{n-1}\Big(\frac12r^2\Big)e^{-\frac14r^2}$$
where $ L_k^{n-1}(r) $ denotes the Laguerre polynomials of type $ (n-1)$. For $\lambda\neq 0$, let $\varphi_{k,\lambda}^{n-1}(r):=\varphi_k^{n-1}(\sqrt{|\lambda|}r).$ By abuse of notation, we write  $\varphi^{n-1}_{k,\lambda}(z):= \varphi_{k,\lambda}^{n-1}(|z|),~z\in \mathbb{C}^n.$
It is well-known that $f^{\lambda}$ has the following expansion (See \cite[2.3.29, Page No.-58]{TH3}) 
\begin{equation}
	\label{flexap}
	f^\lambda(z)  = (2\pi)^{-n} |\lambda|^n \sum_{k=0}^\infty   f^{\lambda}\ast_{\lambda}\varphi^{n-1}_{k,\lambda}(z),
\end{equation}
where $f^{\lambda}\ast_{\lambda}\varphi^{n-1}_{k,\lambda}(z)$ is the $\lambda$-twisted convolution defined by 
$$
f^{\lambda}\ast_{\lambda}\varphi^{n-1}_{k,\lambda}(z)=\int_{\C^n}f^{\lambda}(z-w)\varphi^{n-1}_{k,\lambda}(w)e^{i\frac{\lambda}{2}\Im z\cdot \overline{w}}\,dw.
$$
Now in view of inversion formula for the Fourier transform we have 
$$ f(z,t) = (2\pi)^{-n-1}  \int_{-\infty}^\infty  e^{-i\lambda t} \Big( \sum_{k=0}^\infty   f^{\lambda}\ast_{\lambda}\varphi^{n-1}_{k,\lambda}(z)\Big) |\lambda|^n d\lambda. 
$$
Now using the fact that $(f\ast \mu_r)^{\lambda}(z)=f^{\lambda}\ast_{\lambda}\mu_r(z)$ we see that
\begin{equation}
	\label{fexap}
	f\ast \mu_r(z,t)=\frac{1}{2\pi}\int_{-\infty}^{\infty}e^{-i\lambda t}f^{\lambda}\ast_{\lambda}\mu_r(z) \,d\lambda
\end{equation}  which along with the following expansion proved in   \cite[Theorem 4.1]{TRMI} and \cite[ Proof of Proposition 6.1]{NeT}
$$
f^{\lambda}\ast_{\lambda}\mu_r(z)=  (2\pi)^{-n} |\lambda|^n \sum_{k=0}^{\infty}\frac{k!(n-1)!}{(k+n-1)!}\varphi^{n-1}_{k,\lambda}(r)f^{\lambda}\ast_{\lambda}\varphi^{n-1}_{k,\lambda}(z),
$$
leads  to the expansion 
\begin{equation}
	\label{eq:expression}
	f\ast\mu_r(z,t)=(2\pi)^{-n-1}\int_{-\infty}^{\infty}e^{-i\lambda t}\left(\sum_{k=0}^{\infty}\frac{k!(n-1)!}{(k+n-1)!}\varphi_{k,\lambda}^{n-1}( r)f^{\lambda}\ast_{\lambda}\varphi_{k,\lambda}^{n-1}(z)\right)|\lambda|^n\,d\lambda.
\end{equation}
The above formula which provides a spectral decomposition for the spherical means, will be very useful for our purpose. Next we describe Heisenberg motion group and its connection with the Fourier transform on $\He^n$
\subsection{Heisenberg motion group and Fourier transform}

Let $U(n)$ denote the group of all unitary matrices of order $n$. This acts on $\mathbb{H}^n$ by the automorphisms $$\sigma.(z,t)=(
\sigma z,t),~\sigma\in U(n).$$  We consider the semi-direct product of $\mathbb{H}^n$ and $U(n)$, $G_n:=\mathbb{H}^n \ltimes U(n)$ which acts on $\mathbb{H}^n$ by  $$(z,t,\sigma).(w,s)=\big(z+\sigma w, t+s+\frac12 \Im( z \cdot \overline{\sigma w})\big)$$  whence  the group law in $G_n$ is given by
$$(z,t,\sigma).(w,s,\tau)=\big(z+\sigma w, t+s+\frac12 \Im( z \cdot \overline{\sigma w}), \sigma \tau \big).$$ The group $G_n$ is called the \textit{Heisenberg motion group} which contains $\mathbb{H}^n
$ and $U(n)$ as subgroups. Also $\mathbb{H}^n$ can be identified with the quotient group $G_n/U(n)$. As a matter of fact, functions on $\mathbb{H}^n$ can be viewed as right $U(n)$ invariant functions on $G_n.$ The Haar measure on $G_n$ is given by $d\sigma\,dz\,dt$ where $d\sigma$ denotes the normalised Haar measure on $U(n)$. To bring out the connection between the group Fourier transform on $\mathbb{H}^n$ and the Heisenberg motion group, we need to describe a family of class-$1$ representation of $G_n.$ We start with recalling the definition of such representations.

 Let $G$ be a locally compact topological group and $K$ be a compact subgroup of $G$. Suppose $\pi$ is a representation of $G$ realised on the Hilbert space $H$. Let $H_{K}$ denote the set of all $K$-fixed vectors given by 
 $$H_{K}:=\{v\in H: \pi(k)v=v,~\forall k\in K\}.$$ It can be easily checked the $H_K$ is a subspace of $H$.
  We say that $\pi$ is a class-1 representation of the pair $(G,K)$ if $H_{K}\neq\{0\}.$ Moreover, when $(G, K)$ is a Gelfand pair it is well-known that $\dim H_K=1.$ In the following, we describe certain family of class-1 representation for the Gelfand pair $(G_n, U(n))$. For that we need to set up some more notations. 
  
  For $\alpha\in \mathbb{N}^n$ and $\lambda\neq 0$ let $\Phi^{\lambda}_{\alpha}(x):=|\lambda|^{n/4}\Phi_{\alpha}(\sqrt{|\lambda|}x),~x\in \mathbb{R}^n$ where $\Phi_{\alpha}$ denote the normalised Hermite functions on $\mathbb{R}^n.$ We know that for each $\lambda\neq 0$, $\{\Phi_{\alpha}^{\lambda}: \alpha\in\mathbb{N}^n\}$ forms an orthonormal basis for $L^2(\mathbb{R}^n)$. Suppose 
  $$E^{\lambda}_{\alpha,\beta}(z,t):=(\pi_{\lambda}(z,t)\Phi_{\alpha}^{\lambda}, \Phi_{\beta}^{\lambda}),~(z,t)\in \mathbb{H}^n$$ denotes the matrix coefficients of the Schr\"odinger representation $\pi_{\lambda}$ of $\mathbb{H}^n.$  For each $\lambda\neq0$ and $k\in\mathbb{N}$, we consider the Hilbert space $\mathcal{H}^{\lambda}_k$ spanned by $\{E^{\lambda}_{\alpha,\beta}:\alpha, \beta\in\mathbb{N}^n, |\beta|=k\}$
and equipped with the inner product
$$(f,g)_{\mathcal{H}^{\lambda}_k}:=(2\pi)^{-n}|\lambda|^n\int_{\C^n}f(z,0)\overline{g(z,0)}dz.$$ We define a representation $\rho_k^{\lambda}$ of $G_n$ realised on $\mathcal{H}^{\lambda}_k$ by the prescription
$$\rho_k^{\lambda}(z,t,\sigma)\varphi(w,s):=\varphi((z,t,\sigma)^{-1}(w,s)),~(w,s)\in \mathbb{H}^n.$$ It is well-known that $\rho_k^{\lambda}$ is an irreducible unitary representation of $G_n$ for all $\lambda\neq0$ and $k\in\mathbb{N}.$ Also for $\lambda\neq 0$ and $k\in \mathbb{N}$ we consider the function $e_{k,\lambda}^{n-1}$ on $\mathbb{H}^n$ defined by 
$$e_{k,\lambda}^{n-1}(z,t)= \frac{k!(n-1)!}{(k+n-1)!}\sum_{|\alpha|=k}(\pi_{\lambda}(z,t)\Phi_{\alpha}^{\lambda}, \Phi_{\alpha}^{\lambda}).$$ It is known that the above function can be expressed in terms of Laguerre functions as follows (See \cite[Page No. 52]{TH2})
$$e_{k,\lambda}^{n-1}(z,t)= \frac{k!(n-1)!}{(k+n-1)!}e^{i\lambda t}\varphi_{k,\lambda}^{n-1}(z).$$
It can be checked that $e_{k,\lambda}^{n-1}$ is a $U(n)$-fixed vector corresponding to the representation $\rho_k^{\lambda}$ and hence $\rho_k^{\lambda}$ is a class-$1$ representation of the pair  $(G_n, U(n))$. Moreover,  $(G_n, U(n))$ being a Gelfand pair, $e_{k,\lambda}^{n-1}$ is unique upto scalar multiple. Also it can be easily checked that $e_{k,\lambda}^{n-1}(0,0)=1.$

Given $f\in L^1(\mathbb{H}^n)$, considering it as an $U(n)$-invariant function on $G_n$, we associate an operator valued function $\rho_k^{\lambda}(f)$ acting on $\mathcal{H}^{\lambda}_k$ defined by 
$$\rho_k^{\lambda}(f):=\int_{G_n}f(z,t)\rho_k^{\lambda}(z,t,\sigma)d\sigma\,dz\,dt.$$ 
Now since $\rho_k^{\lambda}$ is unitary, it can be easily checked that $\rho_k^{\lambda}(f)$ is a bounded operator and the operator norm is bounded above by $\|f\|_1.$ As a matter of fact, the scalar valued function $$f\rightarrow \rho_k^{\lambda}(f)e^{n-1}_{k,\lambda}(z,t)=:e^{i\lambda t}\hat{f}(\lambda,k,z)$$ can be viewed as an analogue of Helgason Fourier transform of $f$. We know that  Using the definition of $\rho_k^{\lambda}$ the following can be easily checked: 
$$\rho^{\lambda}_{k}(f)e^{n-1}_{k,\lambda}(z,t)= e^{i \lambda t} f^{-\lambda} \ast_{-\lambda} \varphi_{k,\lambda}^{n-1}(z).$$
  This leads to the following nice formula proved in \cite[Proposition 2.1]{RRT} 
\begin{equation}\label{norm-rho-k}\frac{(k+n-1)}{k! (n-1)!} \| \rho_k^\lambda(f)\|_{HS}^2  = (2\pi)^{-n} |\lambda|^n \,  \int_{\C^n} |f^{-\lambda}\ast_{-\lambda}\varphi^{n-1}_{k,\lambda}(z)  |^2 dz. \end{equation}
		
We end the preliminaries with a description of spectral decomposition of the sublaplacian on $\He^n$  and expansions in terms of Laguerre functions.

 \subsection{The sublaplacian on $\mathbb{H}^n$ and the generalised sublaplacian}
	
	 We let $ \mathfrak{h}_n $ stand for the Heisenberg Lie algebra consisting of left invariant vector fields on $ \mathbb{H}^n .$  A  basis for $ \mathfrak{h}_n $ is provided by the $ 2n+1 $ vector fields
	 $$ X_j = \frac{\partial}{\partial{x_j}}+\frac{1}{2} y_j \frac{\partial}{\partial t}, \,\,Y_j = \frac{\partial}{\partial{y_j}}-\frac{1}{2} x_j \frac{\partial}{\partial t}, \,\, j = 1,2,..., n $$
	 and $ T = \frac{\partial}{\partial t}.$  These correspond to certain one parameter subgroups of $ \mathbb{H}^n.$ The sublaplacian on $\He^n$ is defined by $\mathcal{L}:=-\sum_{j=1}^{\infty}(X_j^2+Y_j^2) $ which is given explicitly by
	 $$\mathcal{L}=-\Delta_{\C^n}-\frac{1}{4}|z|^2\frac{\partial^2}{\partial t^2}+N\frac{\partial}{\partial t}$$ where  $\Delta_{\C^n}$ stands for the Laplacian on $\C^n$ and $N$ is the rotation operator defined by 
	 $$N=\sum_{j=1}^{n}\left(x_j\frac{\partial}{\partial y_j}-y_j\frac{\partial}{\partial x_j}\right).$$ This is a sub-elliptic operator and homogeneous of degree $2$ with respect to the non-isotropic dilation given by $\delta_r(z,t)=(rz,r^2t).$ The sublaplacian is also invariant under  rotation i.e., $$R_{\sigma}\circ \mathcal{L}=\mathcal{L}\circ R_{\sigma},~\sigma\in U(n).$$
	 We denote the full laplacian on $\He^n$ by $\Delta_{\He}$ which is defined as follows:
	 $$\Delta_{\He}=-\sum_{j=1}^{n}(X_j^2+Y_j^2)-T^2.$$
	 We consider the special Hermite operator $L_{\lambda}$ defined by the relation $(\mathcal{L}f)^{\lambda}(z)=L_{\lambda}f^{\lambda}(z).$ It turns out that $ L_\lambda $  is explicitly given by 
	 $$L_\lambda = -\Delta_{\C^n}+\frac{1}{4} \lambda^2 |z|^2+ i \lambda N.$$ In view of the fact that  $f^{\lambda}\ast_{\lambda}\varphi_{k,\lambda}(z)$ are eigenfunctions of $L_{\lambda}$ with eigenvalues $(2k+n)|\lambda|$, using (\ref{flexap}), we have the following expansion 
	 $$L_{\lambda}f^{\lambda}(z)=(2\pi)^{-n} |\lambda|^n \sum_{k=0}^\infty (2k+n)|\lambda|  f^{\lambda}\ast_{\lambda}\varphi_{k,\lambda}^{n-1}(z)$$ 
	 leading to the following spectral decomposition of $\mathcal{L}$:
	 \begin{equation}
	 	\label{lspec}
	 	\mathcal{L}f(z,t)= (2\pi)^{-n-1}  \int_{-\infty}^\infty  e^{-i\lambda t} \Big( \sum_{k=0}^\infty ((2k+n)|\lambda|)   f^{\lambda}\ast_{\lambda}\varphi_{k,\lambda}^{n-1}(z)\Big) |\lambda|^n d\lambda. 
	 \end{equation}	
 Moreover, we can rewrite the Plancherel formual in terms of these projections $f\rightarrow f\ast_{\lambda}\varphi_{k,\lambda}^{n-1}.$ Indeed, it has been proved in \cite[Proposition 2.3.3]{TH3} that 
 $$W_{\lambda}(\varphi_{k,\lambda}^{n-1})=(2\pi)^n\, |\lambda|^{-n} P_k(\lambda).$$ 
Using this and the definition of Hilbert-Schmidt norm we have 
 $$ \|\widehat{f}(\lambda)\|^2_{HS}=\sum_{k=0}^{\infty}\|W_{\lambda}(f^{\lambda})P_k(\lambda)\|^2_{HS}=  (2\pi)^{-2n} |\lambda|^{2n}  \sum_{k=0}^{\infty}\|W_{\lambda}(f^{\lambda}\ast_\lambda  \varphi_{k,\lambda}^{n-1})\|^2_{HS}.$$
In  view of the  Plancherel formula for the Weyl transform \ref{wplan} we get
$$ \int_{\C^n} |f^\lambda(z)|^2 dz =  (2\pi)^{-2n} |\lambda|^{2n} \sum_{k=0}^{\infty}\|f^{\lambda}\ast_\lambda \varphi_{k,\lambda}^{n-1}\|_2^2 .$$
Integrating with respect to $ \lambda $ we obtain
  \begin{equation}\label{projplan}
   	\int_{\He^n} |f(z,t))|^2 dz dt =  (2\pi)^{-2n-1}  \int_{-\infty}^\infty   \big( \sum_{k=0}^{\infty}\|f^{\lambda}\ast_\lambda \varphi_{k,\lambda}^{n-1}\|_2^2 \big) |\lambda|^{2n} \, d\lambda.
  \end{equation}

 We say that a function $ f$ on $\mathbb{H}^n$ is radial if it radial in the $ z $ variable and by abusing the notation we write $ f(z,t) = f(r,t), r =|z|.$ The action of $ \mathcal{L} $ on such radial functions is given by   $\mathcal{L}f(z,t)=\mathcal{L}_{n-1}f(r,t)$ where the operator  $\mathcal{L}_{n-1}$ is given by $$\mathcal{L}_{n-1} = -\frac{\partial^2}{\partial r^2} - \frac{2n-1}{r} \frac{\partial}{\partial r} - \frac{1}{4}r^2  \frac{\partial^2}{\partial t^2}.$$  
 This suggests that we consider the family of operators $ \mathcal{L}_\alpha, \,  \alpha \geq -1/2,$  on $ S = \R^+ \times \R $  defined by
 $$\mathcal{L}_{\alpha} = -\frac{\partial^2}{\partial r^2} - \frac{2\alpha+1}{r} \frac{\partial}{\partial r} - \frac{1}{4}r^2  \frac{\partial^2}{\partial t^2}.$$  
 These operators are called generalized sublaplacians whose spectral decomposition can be written down explicitly. Let us define the Laguerre functions of type $ \alpha \geq -1/2 $ by 
 $$\varphi_{k,\lambda}^{\alpha}(r):=L_k^{\alpha}\Big(\frac12 |\lambda|r^2\Big)e^{-\frac14 |\lambda| r^2}.$$
 It is well known (see \cite{Stem}) that  the functions  $ e^{\alpha}_{k,\lambda}(r,t) $  defined by 
 $$e^{\alpha}_{k,\lambda}(r,t):= \frac{\Gamma(k+1)\Gamma(\alpha +1)}{\Gamma(k+\alpha+1)}e^{i\lambda t}\varphi_{k,\lambda}^{\alpha}(r) $$ are eigenfunctions of $\mathcal{L}_\alpha $ with eigenvalue $(2k+\alpha+1)|\lambda|$ and hence the spectral decomposition of the operator $\mathcal{L}_\alpha$ is then given by 
\begin{equation}\label{lag-exp} 
\mathcal{L}_\alpha f(r,t) = (2\pi)^{-1} \int_{-\infty}^\infty e^{-i\lambda t} \Big(\sum_{k=0}^\infty  (2k+\alpha+1)|\lambda| \, R_{k,\lambda}^\alpha(f )  \varphi_{k,\lambda}^\alpha(r)\Big) d\lambda.
\end{equation}

In the above expansion, the coefficients $ R^{\alpha}_{k,\lambda}(f) $ are given by
$$ R^\alpha_{k,\lambda}(f) =     \int_{-\infty}^{\infty}\int_0^\infty  f(r,t) e_{k,\lambda}^\alpha(r,t) \, r^{2\alpha+1} dr dt.$$  Note that with the obvious definition of $ f^\lambda(r)$ we have
$$ R_{k,\lambda}^\alpha(f) =   \frac{\Gamma(k+1)\Gamma(\alpha +1)}{\Gamma(k+\alpha+1)} \int_0^\infty  f^\lambda(r) \varphi_{k,\lambda}^{\alpha}(r) r^{2\alpha+1} dr.$$
The spectral decomposition (\ref{lag-exp}) leads to the following theorem about expansions in terms of the functions $e_{k,\lambda}^\alpha(r,t).$

\begin{thm}\label{planch}
 For any $ f \in L^2(S, r^{2\alpha+1}dr dt) $ we have the $ L^2$-convergent expansion
$$  f(r,t) =   c_\alpha (2\pi)^{-1} \int_{-\infty}^\infty e^{-i\lambda t} \Big(\sum_{k=0}^\infty   \, R_{k,\lambda}^\alpha(f )  \varphi_{k,\lambda}^\alpha(r)\Big) \, |\lambda|^{\alpha+1} \,d\lambda.$$ 
The Plancherel theorem for the above expansion reads as follows:
$$ \int_{-\infty}^\infty \int_0^\infty  |f(r,t)|^2 r^{2\alpha+1} dr dt = c_\alpha^\prime \int_{-\infty}^\infty \Big(  \sum_{k=0}^\infty  \frac{\Gamma(k+\alpha+1)}{\Gamma(k+1)\Gamma(\alpha +1)}\, |R_{k,\lambda}^\alpha(f)|^2 \Big) |\lambda|^{\alpha+1}\, d\lambda.$$
In the above formulas, $ c_\alpha $ and $ c_\alpha^\prime$ are explicit constants.
\end{thm}
We refer the reader to Stempak \cite{Stem} for more details about the results described above.

In the next section we prove an analogue of Chernoff's theorem for the generalised Laplacian $ \Delta_\alpha = -\partial_t^2+ \mathcal{L}_\alpha$ on $ \R^+ \times \R.$ In view of the expansion (\ref{eq:expression}) the particular case $ \alpha = n-1 $ plays an important role in proving Chernoff's theorem for the sublaplacian on $ \He^n.$

\section{ An analogue of Chernoff's theorem for the Laplacian on $ \He^n.$}

In this section we prove an analogue of Chernoff's theorem for Laplacian $ \Delta_{\He} = -\partial_t^2+\mathcal{L} $ on  $ \He^n.$  As explained earlier, the idea is prove an  analogue of  Chernoff's theorem for the generalised Laplacian $ \Delta_{\alpha} $ first and then use it to deduce the required result.  

\subsection{ Chernoff's theorem for $ \Delta_\alpha$}  As in our earlier works \cite{GT1,GT2} we make use of  the following result of de Jeu  \cite{J}  which is a generalisation of a theorem  of Carleman  in the one  dimensional case.
\begin{thm}
	\label{dj}
	 Let $ \mu $ be a finite positive Borel measure on $ \R^n $ for which all the moments $ M^{(j)}(m) =\int_{\R^n} x_j^m d\mu(x) ,  m \geq 0 $ are finite. If we further assume that the moments satisfy the Carleman condition $ \sum_{m=1}^\infty  M^{(j)}(2m)^{-1/2m} = \infty,\, j=1,2,...,n,$ then polynomials are dense in $ L^p(\R^n,d\mu), 1 \leq p < \infty.$ 
\end{thm} 

\begin{rem} We require the above result only when $ n =2.$ Moreover,  polynomials that are even in the second variable are dense in the space $ L^p_{2,e}(\R^2,d\mu), 1 \leq p < \infty $ consisting of  functions that are even in the second variable.
\end{rem}
We also require the two elementary results about series of positive real numbers described in the following lemma. 

\begin{lem} \label{series} 

(a) 	Let $\{M_n\}_n$ be a sequence of positive real numbers satisfying $\sum_{n=1}^{\infty}M_n^{-1/n}=\infty$. Suppose $\{K_n\}_n$ is another sequence of positive real numbers such that $K_n\leq aM_n+b^n$ for some constants $a,b>0.$ Then $\sum_{n=1}^{\infty}K_n^{-1/n}=\infty.$

(b) Let $\{a_m\}_m$ be a sequence of positive real numbers such that $\sum_{m=1}^{\infty} a_m=\infty$, then for any positive integer $j$, we have $\sum_{m=1}^{\infty} a_m^{1+\frac{j}{m}}=\infty.$
	
\end{lem}

For proofs of the two results stated in the above lemma, we refer the reader to \cite[Lemma 3.2]{C1} and \cite[Lemma 3.3]{BPR} respectively.

We are now in a position to state and prove the following  version of Chernoff's theorem for the operator $ \Delta_\alpha = -\partial_t^2+ \mathcal{L}_\alpha.$ In what follows, we write $L^2(S) $ in place of $  L^2(S, r^{2\alpha+1}dr dt)$ for the sake of brevity.

\begin{thm}
	\label{cl}
	Let $f\in C^{\infty}(S)$ be such that $\Delta_\alpha^mf\in L^2(S) $ for all $m\geq 0$   and satisfies the Carleman condition  $\sum_{m=1}^{\infty}\|\Delta_\alpha^mf\|_{L^2(S)}^{-\frac{1}{2m}}=\infty.$ If $f$ vanishes on a neighbourhood of $(0,0)$, then $f$ is identically zero.
\end{thm}
\begin{proof}  Let $ \widetilde{ \Omega}_\alpha = \{ (\lambda, (2k+\alpha+1)|\lambda|): \lambda \in \R, k \in \mathbb{N} \}, $ which is known as the Heisenberg fan when $ \alpha = n-1.$ We  let $ \Omega_\alpha = \{ (x,y): (x,y^2) \in  \widetilde{ \Omega}_\alpha \} $ and define a measure $\mu_f$ on $ \R^2$  supported on $ \Omega_\alpha $ as follows: for any Borel function $ \varphi$ on $ \R^2$ 
$$ \int_{\R^2}   \varphi(x,y) d\mu_f(x,y) =   \int_{-\infty}^\infty  \Big(\sum_{k=0}^\infty  \varphi_e(\lambda, \sqrt{(2k+\alpha+1)|\lambda|}) \frac{\Gamma(k+\alpha+1)}{\Gamma(k+1)\Gamma(\alpha +1)}\,  |R_{k,\lambda}^\alpha(f )| \Big) |\lambda|^{\alpha+1}d\lambda
$$
 where   $ \varphi_e(x,y) =  \frac{1}{2} \big( \varphi(x,y)+ \varphi(x,-y)  \big).$ Under the assumptions on $ f $ it follows that $ \mu_f $ is a finite Borel measure which satisfies
$$      \int_{\mathbb{R}^2}  \varphi(x,-y) d\mu_f(x,y) = \int_{\mathbb{R}^2} \varphi(x,y) d\mu_f(x,y).$$
As a consequence, all the odd moments $ M^{(2)}(2m+1) $ of $ \mu_f $ are zero and the even moments  are given by 
\begin{equation}\label{eve-mom}
 M^{(2)}(2m) =   \int_{-\infty}^\infty  \Big(\sum_{k=0}^\infty   ((2k+\alpha+1)|\lambda|)^m   \frac{\Gamma(k+\alpha+1)}{\Gamma(k+1)\Gamma(\alpha +1)}\,|R_{k,\lambda}^\alpha(f )| \Big) |\lambda|^{\alpha+1}d\lambda .
 \end{equation} 
 We also have 
 \begin{equation}
 M^{(1)}(2m) =   \int_{-\infty}^\infty  \Big(\sum_{k=0}^\infty  \lambda^{2m}  \frac{\Gamma(k+\alpha+1)}{\Gamma(k+1)\Gamma(\alpha +1)}\,  |R_{k,\lambda}^\alpha(f )| \Big) |\lambda|^{\alpha+1}d\lambda.
 \end{equation}
 We will now show that the moments $ M^{(j)}(2m),\, j=1,2 $ satisfy the Carleman condition. Observe that $ M^{(j)}(2m) \leq M(2m) $ where
 \begin{equation}
 M(2m) =   \int_{-\infty}^\infty  \Big(\sum_{k=0}^\infty \big(\lambda^{2}+(2k+\alpha+1)|\lambda| \big)^m \frac{\Gamma(k+\alpha+1)}{\Gamma(k+1)\Gamma(\alpha +1)}\,  |R_{k,\lambda}^\alpha(f )| \Big) |\lambda|^{\alpha+1}d\lambda.
 \end{equation}
 Therefore, it is enough to check the Carleman condition for $ M(2m).$ By splitting $ M(2m) = M_0(2m)+M_\infty(2m) $ where 
 $$ M_0(2m) =   \int_{-\infty}^\infty  \Big(\sum_{ (2k+\alpha+1)|\lambda| \leq 1}   \big( \lambda^2 + (2k+\alpha+1)|\lambda| \big)^m   \frac{\Gamma(k+\alpha+1)}{\Gamma(k+1)\Gamma(\alpha +1)}\,|R_{k,\lambda}^\alpha(f )| \Big) |\lambda|^{\alpha+1}d\lambda $$
 we estimate them separately.
 
  By applying Cauchy-Schwarz inequality and using the Plancherel formula stated in Theorem \ref{planch} we see that $ M_0(2m)^2 $ bounded by 
 $$  C\,   \|f\|_{L^2(S)}^2  \int_{-\infty}^\infty  \Big(\sum_{ (2k+\alpha+1)|\lambda| \leq 1}   \big(\lambda^2+(2k+\alpha+1)|\lambda| \big)^m   \frac{\Gamma(k+\alpha+1)}{\Gamma(k+1)\Gamma(\alpha +1)}\, \Big) |\lambda|^{\alpha+1}d\lambda .$$
 As $   \frac{\Gamma(k+\alpha+1)}{\Gamma(k+1)\Gamma(\alpha +1)}\, \leq C_\alpha (2k+\alpha+1)^{\alpha} $ and $ \lambda^2+ (2k+\alpha+1)|\lambda| \leq 2 $ the above integral is bounded by
 $$  C_\alpha \, 2^m\,  \sum_{k=0}^\infty  (2k+\alpha+1)^{\alpha} \int_{(2k+\alpha+1)|\lambda| \leq 1} |\lambda|^{\alpha+1} d\lambda \leq C_\alpha^\prime  \,2^m\, \sum_{k=0}^\infty (2k+\alpha+1)^{-2} < \infty .   $$
 This gives the estimate $ M_0(2m) \leq  2^m\, C_1 \|f\|_{L^2(S)} .$  In order to estimate $ M_\infty(2m) $ we choose  a  positive integer $ j > \alpha/2+1$  so that 
 $$  C_j^2 = \int_{-\infty}^\infty  \Big(\sum_{(2k+\alpha+1)|\lambda| \geq 1}    ((2k+\alpha+1)|\lambda|)^{-2j}   \frac{\Gamma(k+\alpha+1)}{\Gamma(k+1)\Gamma(\alpha +1)}\, \Big) |\lambda|^{\alpha+1}d\lambda  < \infty .$$
 By writing 
 $ \big( \lambda^2+ (2k+\alpha+1)|\lambda| \big)^m  = \big( \lambda^2+ (2k+\alpha+1)|\lambda| \big)^{m+j} \big(\lambda^2+ (2k+\alpha+1)|\lambda|\big)^{-j} ,$  using (\ref{lag-exp}) and applying  Cauchy-Schwarz, we see that $ M_\infty(2m)^2 $ is bounded by $ C_j^2 $ times
 $$ \int_{-\infty}^\infty \Big(  \sum_{k=0}^\infty   \big(\lambda^2+ (2k+\alpha+1)|\lambda| \big)^{2(m+j)}  \frac{\Gamma(k+\alpha+1)}{\Gamma(k+1)\Gamma(\alpha +1)}\, |R_{k,\lambda}^\alpha(f)|^2 \Big) |\lambda|^{\alpha+1}\, d\lambda $$ 
 which is a constant multiple of  $ \| \Delta_\alpha^{m+j}f \|_{L^2(S)}^2 .$ Thus we have proved the estimates 
 \begin{equation}\label{carle}
 M^{(j)}(2m) \leq  a_j \| \Delta_\alpha^{m+j}f \|_{L^2(S)} + b^{2m} \|f\|_{L^2(S)} .
 \end{equation}
In view of the second part of Lemma \ref{series}, the hypothesis gives $\sum_{m=1}^{\infty}\|\Delta_\alpha^{m+j}f\|_{L^2(S)}^{-\frac{1}{2m}}=\infty.$ Using this along with the first part of Lemma \ref{series}, the above estimate allows us to  conclude that $\sum_{m=1}^{\infty}M^{(j)}(2m)^{-\frac{1}{2m}}=\infty.$ Thus the moment sequences $ M^{(j)}(2m) $ satisfy the Carleman condition.

Hence by  the remark after Theorem \ref{dj}  we know that polynomials that are even in the second variable are dense in $L^1_{2,e}(\R^2,d\mu_f).$ Consider the function $\varphi $ defined on $ \Omega_\alpha$ by 
$$  \varphi(\lambda, \sqrt{(2k+\alpha+1)|\lambda|}) =  \varphi(\lambda, -\sqrt{(2k+\alpha+1)|\lambda|}) =   \overline{R_{k,\lambda}^\alpha (f)}.$$
As $ \varphi(x,y) $ is even in the second variable it follows that
$$\int_{\R^2} |\varphi(x,y)| d\mu_f(x,y) = \int_{-\infty}^\infty \Big(  \sum_{k=0}^\infty  \frac{\Gamma(k+\alpha+1)}{\Gamma(k+1)\Gamma(\alpha +1)}\, |R_{k,\lambda}^\alpha(f)|^2 \Big) |\lambda|^{\alpha+1}\, d\lambda = c_\alpha^{-1} \|f\|_{L^2(S)}^2.$$
This shows that $ \varphi \in L^1_{2,e}(\R^2,d\mu_f) $ and hence given any $ \epsilon > 0 $ we can find a polynomial $ q(x,y) $ which is even in the second variable such that
$$ \left| \int_{\R^2} (\varphi(x,y) - q(x,y)) d\mu_f(x,y) \right| \leq  \int_{\R^2} |\varphi(x,y)-q(x,y)|d\mu_f(x,y) < \epsilon.$$
Therefore, with $ \psi(x,y) = \varphi(x,y)-q(x,y), $ which is even in the second variable, we have 
\begin{equation}\label{approx} \left| \int_{-\infty}^\infty  \Big(   \sum_{k=0}^\infty  \psi(\lambda, \sqrt{(2k+\alpha+1)|\lambda|}) \frac{\Gamma(k+\alpha+1)}{\Gamma(k+1)\Gamma(\alpha +1)}\, R_{k,\lambda}^\alpha(f)  \Big) |\lambda|^{\alpha+1} d\lambda \right| < \epsilon.
\end{equation}
We now claim that 
$$  \int_{-\infty}^\infty  \Big(   \sum_{k=0}^\infty  q(\lambda, \sqrt{(2k+\alpha+1)|\lambda|}) \frac{\Gamma(k+\alpha+1)}{\Gamma(k+1)\Gamma(\alpha +1)}\, R_{k,\lambda}^\alpha(f)  \Big) |\lambda|^{\alpha+1} d\lambda  = 0.$$
Assuming the claim for a moment let us complete the proof. Recalling the definition of $ \varphi ,$ from (\ref{approx}) we obtain 
$$ \int_{-\infty}^\infty  \Big(   \sum_{k=0}^\infty  \varphi(\lambda, \sqrt{(2k+\alpha+1)|\lambda|}) \frac{\Gamma(k+\alpha+1)}{\Gamma(k+1)\Gamma(\alpha +1)}\, R_{k,\lambda}^\alpha(f)  \Big) |\lambda|^{\alpha+1} d\lambda  = 
c_\alpha^{-1}  \|f\|_{L^2(S)}^2 < \epsilon.$$
As $ \epsilon $ is arbitrary, this proves the theorem.

Returning to the claim, it is enough to show that 
\begin{equation}\label{claim-2}  \int_{-\infty}^\infty  \Big(   \sum_{k=0}^\infty  \lambda^j  ((2k+\alpha+1)|\lambda|)^m  \frac{\Gamma(k+\alpha+1)}{\Gamma(k+1)\Gamma(\alpha +1)}\, R_{k,\lambda}^\alpha(f)  \Big) |\lambda|^{\alpha+1} d\lambda  = 0 
\end{equation}
for any $ j, m  \in \mathbb{N}.$ This follows from the hypothesis that $ f $ vanishes in a neighbourhood of $ (0,0) $ and the inversion formula ( see Theorem \ref{planch}):

$$  f(r,t) =  c_\alpha (2\pi)^{-1} \int_{-\infty}^\infty  \Big(\sum_{k=0}^\infty   \, \frac{\Gamma(k+\alpha+1)}{\Gamma(k+1)\Gamma(\alpha +1)}\,R_{k,\lambda}^\alpha(f )   e_{k,-\lambda}^\alpha(r, t)\Big) \, |\lambda|^{\alpha+1} \,d\lambda.$$
By applying $ \partial_t^j\, \mathcal{L}_\alpha^m $ to the above formula, evaluating at $ (0,0) $ and using $ e_{k,\lambda}^\alpha(0,0) = 1 $ the claim (\ref{claim-2}) is proved.
\end{proof}

\subsection{Chernoff's theorem for $ \Delta_{\He}$ on the Heisenberg group} 	

 We make use of this Theorem \ref{cl} to prove the following analogue of Chernoff's theorem for the  full Laplacian $\Delta_{\He}$ on $\mathbb{H}^n.$  For the proof we need the expansion (\ref{eq:expression}) of the spherical means $ f \ast \mu_r(z,t) $ in terms of $ \varphi_{k,\lambda}^{n-1}(t).$
 \begin{thm}
 	\label{csl}
 	Let $f\in C^{\infty}(\mathbb{H}^n)$ be such that $\Delta_{\He}^mf\in L^2(\mathbb{H}^n)$ for all $m\geq 0$  and satisfies the Carleman condition $\sum_{m=1}^{\infty}\|\Delta_{\He}^mf\|_2^{-\frac{1}{2m}}=\infty.$ If $f$ vanishes on a non-empty open set, then $f$ is identically zero.
 \end{thm}
\begin{proof}
	Since the Laplacian $\Delta_{\He}$ is translation invariant, without loss of generality we can assume that $f$ vanishes on an open neighbourhood $V$ of the identity in $\mathbb{H}.$  Clearly for some $ a>0,\, B_{a}(0)\times (-a,a)\subset V$ where $B_a(0)$ denotes the ball of radius $a$ in $\mathbb{C}^n.$ Now we consider the spherical means of $f$ 
	$$f\ast\mu_r(z,t):=\int_{|w|=r}f\Big(z-w,t-\frac12\Im z\cdot \overline{w}\Big)\,d\mu_r(w)$$ 
and we consider  $F_{z}(r,t):=f\ast\mu_r(z,t) $ as a function on $ S = \R^+ \times \R.$ Let $ \delta = \min(a/2, \sqrt{a})$. For any  $z\in B_{\delta}(0),\, (r,t)\in U:=(0,\delta)\times (-\delta/2,\delta/2)$ and   $|w|=r$, we see that $|z-w|< a$ and $|t-\frac12\Im z\cdot \overline{w}|< a/2 + \delta^2/2\leq a $ so  that $(z-w,t-\frac12\Im z\cdot \overline{w})\in B_a(0)\times (-a,a)$. Consequently, for any  $z\in B_{\delta}(0)$, $F_z(r,t)=0$ for all $(r,t)\in U.$ 
Now,  a comparison of Theorem \ref{planch} with the following expansion
\begin{equation}
	\label{eq:express}
 F_z(r,t) =(2\pi)^{-n-1}\int_{-\infty}^{\infty}e^{-i\lambda t}\left(\sum_{k=0}^{\infty}\frac{k!(n-1)!}{(k+n-1)!}\varphi_{k,\lambda}^{n-1}( r)f^{\lambda}\ast_{\lambda}\varphi_{k,\lambda}^{n-1}(z)\right)|\lambda|^n\,d\lambda
\end{equation}
shows that $$ R_{k,\lambda}^{n-1}(F_z) =  \frac{k!(n-1)!}{(k+n-1)!}  f^\lambda \ast_\lambda \varphi_{k,\lambda}^{n-1}(z).$$ 

As $ \varphi_{k,\lambda}^{n-1}(r) e^{-i\lambda t} $ are eigenfunctions of $ \Delta_{n-1} = -\partial_t^2+\mathcal{L}_{n-1}$ it follows from the Plancherel  formula in Theorem \ref{planch} that
\begin{equation}\label{scl1}
		\|\Delta_{n-1}^mF_z\|_{L^2(S)}^2=   \int_{-\infty}^{\infty}  \Big(  \sum_{k=0}^\infty ( \lambda^2+(2k+n)|\lambda|)^{2m}\,   \frac{k!(n-1)!}{(k+n-1)!}\, |f^{\lambda}\ast_\lambda\varphi_{k,\lambda}^{n-1}(z)|^2 \Big)|\lambda|^nd\lambda.
\end{equation}
In view of  the well  known  formula $  \varphi_{k,\lambda}^{n-1} \ast_\lambda \varphi_{k,\lambda}^{n-1} = (2\pi)^{n} |\lambda|^{-n} \varphi_{k,\lambda}^{n-1} $ (see \cite[Corollary 2.3.4]{TH3}
we have  $$ f^\lambda \ast_\lambda \varphi_{k,\lambda}^{n-1} (z)= (2\pi)^{-n} |\lambda|^{n}\,  f^\lambda \ast_\lambda \varphi_{k,\lambda}^{n-1} \ast_\lambda  \varphi_{k,\lambda}^{n-1}(z) $$
which gives the following estimate by Cauchy-Schwarz inequality:
\begin{equation}\label{hsnorm}  |\lambda|^{n} \, |f^\lambda \ast_\lambda \varphi_{k,\lambda}^{n-1}(z)|^2 \leq c_n \, |\lambda|^{2n}\, \frac{(k+n-1)!}{k!(n-1)!} \| f^\lambda \ast_\lambda \varphi_{k,\lambda}^{n-1}\|_2^2.
\end{equation}
In proving the above we have made use of the fact 
 $\|\varphi_{k,\lambda}^{n-1} \|^2_2=c_n|\lambda|^{-n} \frac{(k+n-1)!}{k!(n-1)!}.$ Using this in (\ref{scl1}) and recalling the Plancherel formula (\ref{projplan}) we obtain

	$$	\|\Delta_{n-1}^mF_z\|_{L^2(S)}^2\leq c_n  \int_{-\infty}^{\infty}  \Big( \sum_{k=0}^\infty (\lambda^2+ (2k+n)|\lambda|)^{2m}  \| f^\lambda \ast_\lambda \varphi_{k,\lambda}^{n-1}\|_2^2 \Big) |\lambda|^{2n}\, d\lambda = c_n \| \Delta_{\He}^m f \|_2^2 $$

Therefore,  the hypothesis on $ f $ allows us to  conclude that $\sum_{m=1}^{\infty}\|\Delta_{n-1}^mF_z\|_{L^2(S)}^{-\frac1{2m}}=\infty.$   But we know  that $F_z$ vanishes on the  neighbourhood $U$ of $(0,0)$  and hence we can appeal to  Theorem \ref{cl} to conclude that $F_z$ is identically zero.
This means for any $z\in B_{\delta}(0)$, $f^{\lambda}\ast_{\lambda}\varphi_{k,\lambda}^{n-1}(z)=0$ for every $(\lambda, k)\in \mathbb{R}\times \mathbb{N}.$ But $f^{\lambda}\ast_{\lambda}\varphi_{k,\lambda}^{n-1},$ being an eigenfunction of the elliptic operator $L_{\lambda}$, is real analytic. Hence $f^{\lambda}\ast_{\lambda}\varphi_{k,\lambda}^{n-1}$ is identically zero for all $\lambda$ and $k$. Therefore, it follows that $f=0$ which proves the theorem.
\end{proof}

\begin{rem}\label{rem-cher} A close examination of the above proof shows that we only need to assume that 
$ \sum_{k=0}^\infty \|\Delta_{n-1}^mF_z\|_{L^2(S)}^{-\frac{1}{2m}} = \infty $  for all $ z \in B_\delta(0) .$
We will make use of this observation in formulating and proving an Ingham type theorem for the Fourier transform on the Heisenberg group in the next section.
\end{rem}

\begin{rem}
		Proving the exact analogue of Chernoff's theorem for $\Delta_{\He}$ where the vanishing condition in Theorem \ref{csl} is replaced by the vanishing of all partial derivatives of the function at a single point, is a very interesting open problem.  
\end{rem}
\section{ Ingham's theorem on the Heisenberg group}
	
	In this section, we make use of the version of Chernoff's theorem proved in the previous section, to prove Ingham type uncertaity principles on  $\mathbb{H}^n$. 
\subsection{ Ingham's theorem for the Fourier transform}  We begin with a proof of Theorem \ref{ingh-easy}. Under the integrability assumption on $ \Theta $ we can construct compactly supported functions $ g$ and $ h $ on $ \He^n$ and $ \R $ respectively such that
$$  \hat{g}(\lambda)^\ast \hat{g}(\lambda) \leq C \, e^{-2\sqrt{H(\lambda)} \,\Theta(\sqrt{H(\lambda))}},\,\,  |\hat{h}(\lambda)| \leq C e^{-|\lambda|\, \Theta(|\lambda|)}.$$ 
(See \cite[Theorem 4.5]{BGST} for the Heisenberg group case and \cite{I} for $\mathbb{R}$). Then the function $  f = g \ast_3 h $ satisfies (\ref{Ing-decay-mod}) where $\ast_3$ stands for the convolution in the $t$-variable . For the converse, assume that $ \int_1^\infty  \Theta(t) t^{-1} dt = \infty.$  If $ f $ is compactly supported and satisfies (\ref{Ing-decay-mod}) we need to  prove that $ f =0.$ We first assume that  $ \Theta(\lambda) \geq c \lambda^{-1/2},\, \lambda \geq 1.$ It is enough to show that for any $ \varphi \in L^2(\C^n)$ the function $  f_\varphi(t) = \int_{\C^n} f(z,t) \varphi(z) dz  $ vanishes identically. As $ f_\varphi $ is compactly supported, in view of Ingham's theorem for the Fourier transform on $ \R $ it is enough to show that $ |\hat{f_\varphi}(\lambda)| \leq C e^{-|\lambda|\, \Theta(|\lambda|)}.$ By Cauchy-Schwarz inequality,
$$ |\hat{f_\varphi}(-\lambda)| = |\int_{C^n} f^\lambda(z) \varphi(z) dz | \leq \|\varphi \|_2  \| f^\lambda\|_2 = (2\pi)^{-n/2} \|\varphi \|_2 |\lambda|^{n/2} \| \hat{f}(\lambda)\|_{HS} .  $$
Calculating the Hilbert-Schmidt norm by using the Hermite basis and using the hypothesis (\ref{Ing-decay-mod}) we obtain
$$  \| \hat{f}(\lambda)\|_{HS}^2 \leq C  e^{-2 |\lambda|\, \Theta(|\lambda|)} \big( \sum_{k=0}^\infty \frac{(k+n-1)!}{k! (n-1)!} e^{-2 \sqrt{(2k+n)|\lambda|}\, \Theta(\sqrt{(2k+n)|\lambda|})} \big).$$
Under the extra assumption on $ \Theta $ the above sum is bounded by a constant multiple of 
$$   \big( \sum_{k=0}^\infty (2k+n)^{n-1}  e^{-2c ((2k+n)|\lambda|)^{1/4}}  \big) \leq  C_1 \int_0^{\infty} t^{n-1} e^{-2\,c \,(|\lambda|t)^{1/4}} \leq C_2 |\lambda|^{-n}. $$
This proves the required estimate on $ \hat{f_\varphi}(\lambda)$ under the extra assumption on $ \Theta.$

The extra assumption on $ \Theta $ namely, $ \Theta(\lambda) \geq c \lambda^{-1/2},\, \lambda \geq 1$ can be removed by proceeding as in \cite[Theorem 4.6]{BGST}. With $ \theta(\lambda)= (1+\lambda^2)^{-1/4} $ we can construct a  compactly supported radial function $ g $ on $ \He^n $  such that
$$  \hat{g}(\lambda)^\ast \hat{g}(\lambda) \leq C e^{-2 \sqrt{H(\lambda)} \, \theta(\sqrt{(H(\lambda)})}\,  e^{-2|\lambda| \, \theta(|\lambda|)} $$
and let $ g_\delta(z,t) = \delta^{-(2n+2)} g(\delta^{-1} z, \delta^{-2}t) .$  Then as  shown in  [1,  Theorem 4.6] the function
$ f_\delta(z,t) = f \ast g_\delta(z,t) $ will satisfy the hypothesis with $\Theta $ replaced by  $ \Psi_\delta(\lambda) = \Theta(\lambda)+  \theta_\delta(\lambda) $  for which the extra condition viz. $ \Psi_\delta(\lambda) \geq c_\delta |\lambda|^{-1/2}, |\lambda| \geq 1 $ holds. Hence, we can conclude that $ f \ast g_\delta = 0 $ for all $ \delta >0.$ Finally an approximate identity argument completes the proof.\\

We remark that the above proof does not work if  $ f $ is not compactly supported but only vanishes on an open  set. This is simply because the function $ f_\varphi(t) $ need not vanish on any open interval. We now present a proof of Theorem \ref{ingh-hei-1} for which we require the following preparatory lemma and a proposition.

\begin{lem}\label{ing-lem} Let $ \Theta $ be as in Theorem \ref{ingh-hei-1}. Further assume that $ \Theta(\lambda) \geq c \lambda^{-1/2},\, \lambda \geq 1.$ Then the sequence 
$ A_m = \int_0^\infty \lambda^{m+n}  e^{-\lambda \,\Theta(\lambda)} d\lambda $ satisfies the estimate $ A_m \leq C_n \big( \frac{2m}{\Theta(m^4)}\big)^m $ for all $ m \geq m_0.$
\end{lem}

This is proved as part of the proof of Ingham's theorem in \cite{I}. We also need the following proposition proved in \cite[Proof of Theorem 4.6]{BGST}.

\begin{prop}\label{cher-prop} Let $ \Theta $ be as in Theorem \ref{ingh-hei-1}. Further assume that $ \Theta(\lambda) \geq c \lambda^{-1/2},\, \lambda \geq 1.$ Then  under the assumption that $  \hat{f}(\lambda)^\ast \hat{f}(\lambda) \leq C \, e^{-2\sqrt{H(\lambda)} \,\Theta(\sqrt{H(\lambda))}},$ then for some constant $ a \geq 1 $ we have the estimate
$ \| \mathcal{L}^m f\|_2  \leq \big( \frac{ a m}{\Theta(m^4)}\big)^{2m} $ for all $  m \geq m_0. $
\end{prop}

{\bf{Proof of Theorem \ref{ingh-hei-1}:}} First we make the observation: without loss of generality, we can assume that $f$  vanishes on $B_{\mathbb{H}}(0,a)$ for some $a>0$ where $B_{\mathbb{H}}(0,a)$ is  the Koranyi ball of radius $a$.  We first assume that  $ \Theta(\lambda) \geq c \lambda^{-1/2},\, \lambda \geq 1.$  We will show that under the hypothesis in Theorem \ref{ingh-hei-1}, the function $ f $ satisfies the conditions of  Theorem \ref{cher}. In view of Plancherel theorem for $ \He^n$ we have
$$ \| \Delta_{\He}^m f \|_2^2 = (2\pi)^{-n-1} \int_{-\infty}^\infty  \|  \hat{f}(\lambda)\big(\lambda^2+ H(\lambda)\big)^m \|_{HS}^2  \, |\lambda|^n d\lambda.$$ 
Calculating the Hilbert-Schmidt operator norm using the Hermite basis,  we have 
$$\|\Delta_{\He}^mf\|^2_2=(2\pi)^{-n-1}\int_{-\infty}^{\infty} \big(\sum_{\alpha \in \mathbb{N}^n} \,(\lambda^2+(2|\alpha|+n)|\lambda|)^{2m}\|\hat{f}(\lambda)\Phi^{\lambda}_{\alpha}\|^2_2 \big) \, |\lambda|^n\, d\lambda.$$ 
In estimating the above we split the sum into two parts. The term where  the sum is taken over those $ \alpha $ for which $ (2|\alpha|+n) \geq |\lambda| $ is bounded by
\begin{equation}\label{est-1}   2^{2m} (2\pi)^{-n-1}\int_{-\infty}^{\infty} \big(\sum_{(2|\alpha|+n) \geq |\lambda| } \,((2|\alpha|+n)|\lambda|)^{2m}\|\hat{f}(\lambda)\Phi^{\lambda}_{\alpha}\|^2_2 \big) \, |\lambda|^n\, d\lambda \leq 2^{2m} \| \mathcal{L}^m f\|_2^2 .
\end{equation} 
The remaining part of $ \|\Delta_{\He}^mf\|^2_2$ is bounded by
\begin{equation}\label{est-2}
 2^{2m} (2\pi)^{-n-1}\int_{-\infty}^{\infty}   \lambda^{4m} \big(\sum_{(2|\alpha|+n) \leq |\lambda| } \,\|\hat{f}(\lambda)\Phi^{\lambda}_{\alpha}\|^2_2 \, \big) \, |\lambda|^n\, d\lambda.
 \end{equation}
Under the hypothesis on $ f $ the Fourier transform satisfies (\ref{Ing-decay-mod}) and hence the above term is bounded by
$$ 2^{2m} (2\pi)^{-n-1}\int_{-\infty}^{\infty}   \lambda^{4m}  e^{-2 |\lambda|\,\Theta(|\lambda|)} \big(\sum_{(2|\alpha|+n) \leq |\lambda| } \, e^{-2 \sqrt{(2|\alpha+n)|\lambda|}  \, \Theta( \sqrt{(2|\alpha|+n)|\lambda|})} \big) \, |\lambda|^n\, d\lambda. $$
Under the extra assumption, $ \lambda \, \Theta(\lambda) \geq c \lambda^{1/2}$, the sum inside the above integral is bounded by
$$  \sum_{(2|\alpha|+n) \leq |\lambda| } \, e^{-2 c ((2|\alpha+n)|\lambda|)^{1/4}}  \leq \sum_{\alpha \in \mathbb{N}^n} \, e^{-2c \sqrt{(2|\alpha|+n)}} \leq C. $$
Thus, the term (\ref{est-2}) is estimated by the integral
\begin{equation}\label{est-3}
C\, 2^{2m} \, \int_{-\infty}^{\infty} \lambda^{4m}|\lambda|^n  e^{-2 |\lambda|\,\Theta(|\lambda|)} d\lambda  \leq C_n 2^{-2m} \,\int_{0}^{\infty} \lambda^{4m+n}  e^{- \lambda \,\Theta(\lambda)} d\lambda.
\end{equation}

We can therefore estimate (\ref{est-1}) by using Proposition \ref{cher-prop} and (\ref{est-3}) by means of Lemma \ref{ing-lem} and  for large $m$ obtain
\begin{equation}
\| \Delta_{\He}^m f \|_2 \leq 2^m\, \big( \frac{ a\, m}{\Theta(m^4)}\big)^{2m}  + C_n 2^{-m} \,  \big( \frac{2m}{\Theta(m^4)}\big)^{2m} \leq C^{2m} \big( \frac{m}{\Theta(m^4)}\big)^{2m}
\end{equation}
for some constant $ C >0.$ As $ t^{-1} \Theta(t) $ is not integrable over $ [1,\infty)  $ it follows that $ \sum_{m=1}^\infty  \frac{\Theta(m^4)}{m} = \infty $ and hence $ f $ satisfies the hypothesis in Theorem \ref{cher}. Consequently, $ f $ vanishes identically.

This proves the theorem under the extra assumption on  $ \Theta.$  The general case can be proved as in the proof of Theorem \ref{ingh-easy}  presented above after some suitable modifications at certain places. Indeed, take $ \theta(\lambda)= (1+\lambda^2)^{-1/4} $.  As explained at the beginning of the proof of Theorem \ref{ingh-easy} above,  we can construct a  compactly supported radial function $ g $ on $ \He^n $  such that
	$$  \hat{g}(\lambda)^\ast \hat{g}(\lambda) \leq C e^{- 2 \sqrt{H(\lambda)} \, \theta(\sqrt{(H(\lambda)})}\,  e^{-2 |\lambda| \, \theta(|\lambda|)} $$ and further we can arrange that $ \text{supp}(g)\subset B_{\mathbb{H}}(0,a/2).$ Now defining $g_{\delta}$ as in the proof of Theorem \ref{ingh-easy}, we observe that $f_{\delta}:=f\ast g_{\delta}$ vanishes on $B_{\mathbb{H}}(0,\delta a/2)$ for all $0<\delta<1$. Moreover, as shown in   \cite[Theorem 4.6]{BGST} the function
	$ f_\delta $ will satisfy the hypothesis with $\Theta $ replaced by  $ \Psi_\delta(\lambda) = \Theta(\lambda)+  \theta_\delta(\lambda) $  for which the extra condition viz. $ \Psi_\delta(\lambda) \geq c_\delta |\lambda|^{-1/2}, |\lambda| \geq 1 $ holds. Hence, we can conclude that $ f \ast g_\delta = 0 $ for all $ 0<\delta<1.$ Finally using an approximate identity argument, letting $\delta$ go to zero, we obtain $f=0$ which proves the theorem.

\subsection{Ingham's theorem for the spectral projections} An examination of the above proof reveals that we do not need the full power of the hypothesis (\ref{Ing-decay-mod}) in proving Theorem \ref{ingh-hei-1}. In fact, it is sufficient to assume that for every $ k $
\begin{equation}\label{hs-norm}   \sum_{|\alpha|=k} \| \hat{f}(\lambda)\Phi_\alpha^\lambda \|_2^2 \leq C e^{-2|\lambda|\,\Theta(|\lambda|)} e^{- 2\sqrt{(2k+n)|\lambda|}\, \Theta(\sqrt{(2k+n)|\lambda|})} .\end{equation}
The sum in the above  is just $ \| \hat{f}(\lambda)P_k(\lambda)\|_{HS}^2 $ and since $ W_\lambda(f^\lambda \ast_\lambda \varphi_{k,\lambda}^{n-1}) = (2\pi)^{n} |\lambda|^{-n} \hat{f}(\lambda)P_k(\lambda),$ the estimate (\ref{hs-norm}) follows once we assume that
\begin{equation}\label{hs-norm-1} |\lambda|^n\,\int_{\C^n} | f^\lambda \ast_\lambda \varphi_{k,\lambda}^{n-1}(z)|^2 dz \leq C\,  e^{-2|\lambda|\,\Theta(|\lambda|)} e^{-2 \sqrt{(2k+n)|\lambda|}\, \Theta(\sqrt{(2k+n)|\lambda|})} .\end{equation}
In view of the formula (\ref{norm-rho-k}) it is clear that (\ref{hs-norm-1}) is an immediate consequence of 
\begin{equation}\label{hs-norm-2}   \frac{(k+n-1)!}{k!(n-1)!} \, \|\rho_k^\lambda(f)\|_{HS}^2  \leq C\,  e^{-2|\lambda|\,\Theta(|\lambda|)} e^{- 2 \sqrt{(2k+n)|\lambda|}\, \Theta(\sqrt{(2k+n)|\lambda|})} .\end{equation}
These chain of inequalities clearly show that Theorem \ref{ingh-hei-1} can be proved  under the assumption (\ref{norm-rho-k}) as claimed in the introduction. We now present a proof of Theorem \ref{ingh-hei-2} which shows that the norm estimate on $ \rho_k^\lambda $ can be replaced by pointwise estimate.\\

{\bf{Proof of Theorem \ref{ingh-hei-2}:}}
  If we let $ f_h(g) = f(h^{-1}g) $ stand for  the left translation of $ f $ by an element $ h $ of $ \He^n$, then $ \rho_k^\lambda(f_h)e_{k,\lambda}^{n-1}(z,t) = \rho_k^\lambda(f)e_{k,\lambda}^{n-1}(h^{-1}(z,t)) $ and hence we can assume that $ f $ vanishes on a neighbourhood of $0.$ Without loss of generality we can assume that $ f $ vanishes in a neighbourhood $ V $ of  zero .  In view of  Remark \ref{rem-cher} it is enough to show that $ \sum_{m=1}^\infty \|\Delta_{n-1}^mF_z\|_{L^2(S)}^{-\frac{1}{2m}} =\infty $ for all $ z \in B_\delta(0) $ for some $\delta>0$ where $F_{z}$ is as in the proof of Theorem \ref{csl} and
\begin{equation}\label{scl2}
		\|\Delta_{n-1}^mF_z\|_{L^2(S)}^2=   \int_{-\infty}^{\infty}  \Big(  \sum_{k=0}^\infty ( \lambda^2+(2k+n)|\lambda|)^{2m}\,   \frac{k!(n-1)!}{(k+n-1)!}\, |f^{\lambda}\ast_\lambda\varphi_{k,\lambda}^{n-1}(z)|^2 \Big)|\lambda|^nd\lambda.
\end{equation}
We can rewrite the above in terms of  $ \rho_k^\lambda(f)e_{k,\lambda}^{n-1}(z,t) $ using  the relation
$$ \rho_k^\lambda(f)e_{k,\lambda}^{n-1}(z,t) =  e^{i \lambda t} f^{-\lambda} \ast_{-\lambda} \varphi_{k,\lambda}^{n-1}(z) .$$ 
Thus we are led to estimate the following:
$$   \int_{-\infty}^{\infty}  \Big(  \sum_{k=0}^\infty ( \lambda^2+(2k+n)|\lambda|)^{2m}\,   \frac{k!(n-1)!}{(k+n-1)!} \, \sup_{(z,t) \in V} |\rho_k^{-\lambda}(f)e_{k,\lambda}^{n-1}(z,t)|^2 \Big)|\lambda|^n d\lambda $$
which under the assumption that 
$$ \sup_{(z,t) \in V}  |\rho_k^\lambda(f)e_k^\lambda(z,t)| \leq C\, e^{-|\lambda|\,\Theta(|\lambda|)} e^{-\sqrt{(2k+n)|\lambda|}\, \Theta(\sqrt{(2k+n)|\lambda|})},~\forall \lambda,k$$ along with (\ref{scl2}) shows that $	\|\Delta_{n-1}^mF_z\|_{L^2(S)}^2$ is dominated by 
$$  \int_{-\infty}^{\infty}  \Big(  \sum_{k=0}^\infty ( \lambda^2+(2k+n)|\lambda|)^{2m}\,   \frac{k!(n-1)!}{(k+n-1)!}\, e^{-2|\lambda|\,\Theta(|\lambda|)} e^{-2\sqrt{(2k+n)|\lambda|}\, \Theta(\sqrt{(2k+n)|\lambda|})} \Big)|\lambda|^nd\lambda. $$   Now under the assumption that $ \Theta(\lambda) \geq c \lambda^{-1/2},\, \lambda \geq 1$, as in the proof of Theorem \ref{ingh-hei-1},  we can show that  $	\|\Delta_{n-1}^mF_z\|_{L^2(S)}$ satisfies the Carleman condition and hence by Theorem \ref{cher} we conclude that $f$ is identically zero. 
 
For the general case, we proceed as follows. Let $g_{\delta}$ and $f_{\delta}$ be as in the proof of Theorem \ref{ingh-hei-1}. Then we have $f_{\delta}$ vanishes in a neighbourhood $V_{\delta}$ of the origin for all $0<\delta<1.$ We need to show that $f_{\delta}$ satisfies the hypothesis of the Theorem \ref{ingh-hei-2}.  Since $g_{\delta}$ is radial, it follows that $$|R^{n-1}_{k,\lambda}(g_{\delta})|\leq C e^{-|\lambda|\,\theta_{\delta}(|\lambda|)} e^{-\sqrt{(2k+n)|\lambda|}\, \theta_{\delta}(\sqrt{(2k+n)|\lambda|})}, ~\text{for all}\ \lambda, \ k$$
 where $$R^{n-1}_{k,\lambda}(g_{\delta})=\frac{k!(n-1)!}{(k+n-1)!}\int_{C^n}g^{\lambda}_{\delta}(z)\varphi_{k,\lambda}^{n-1}(z)dz.$$ Now expanding $g^{\lambda}_{\delta} $ in terms of Laguerre functions (See \cite[Proof of Proposition 2.4.2]{TH3} ) and  making use of the following fact (See \cite[Corollary 2.3.4]{TH3}) $$\varphi_{k,\lambda}^{n-1}\ast_{\lambda}\varphi_{m,\lambda}^{n-1}=\delta_{km}(2\pi)^{n}|\lambda|^{-n}\varphi_{k,\lambda}^{n-1}$$ we obtain 
\begin{equation}
	\rho_k^{\lambda}(f_{\delta})e^{n-1}_{k,\lambda}(z,t)= \,e^{i\lambda t} \, R^{n-1}_{k,-\lambda}(g_{\delta}) \,\,f^{-\lambda}\ast_{-\lambda}\varphi_{k,\lambda}^{n-1}(z).
\end{equation}
Hence it follows that $$\sup_{(z,t) \in V_{\delta}}|\rho_k^{\lambda}(f_{\delta})e^{n-1}_{k,\lambda}(z,t)|\leq C e^{-|\lambda|\,\Psi_{\delta}(|\lambda|)} e^{-\sqrt{(2k+n)|\lambda|}\, \Psi_{\delta}(\sqrt{(2k+n)|\lambda|})}$$ where $\Psi_{\delta}:=\Theta+\theta_{\delta}$ and by construction $\Psi_{\delta}(\lambda)\geq c_{\delta}|\lambda|^{-1/2}$ for $|\lambda|\geq 1.$ Therefore, from the first part of the proof it follows that $f_{\delta}=0$ for $0<\delta<1$ which in view of an approximate identity type argument yields $f=0$ proving the theorem.
\begin{rem}
	The Theorem \ref{ingh-hei-2} is sharp in the sense that when $\int_{1}^{\infty}\Theta(t)t^{-1}dt<\infty$ there exist a compactly supported smooth function $f$ on $\mathbb{H}^n$ satisfying the uniform estimate
	\begin{equation}
		\label{dcingh1}
		|\rho_k^\lambda(f)e_k^\lambda(z,t)| \leq C\, e^{-|\lambda|\,\Theta(|\lambda|)} e^{-\sqrt{(2k+n)|\lambda|}\, \Theta(\sqrt{(2k+n)|\lambda|})}.
	\end{equation}
	 Indeed, as explained in the proof of Theorem \ref{ingh-easy} there exist a compactly supported smooth radial function $f$ on $\mathbb{H}^n$ whose Fourier transform satisfies (\ref{Ing-decay-mod}). Now since $f$ is radial, proceeding as in the proof above, the above estimate (\ref{dcingh1}) can be checked easily. 
\end{rem}
\begin{rem}
	It would be interesting to see whether the conclusions of the Theorems \ref{ingh-easy}, \ref{ingh-hei-1} and \ref{ingh-hei-2} still hold true if we use two different decreasing functions in the decay condition  instead of just one.  
\end{rem}

\section*{Acknowledgments}   The first author is supported by Int. Ph.D. scholarship from Indian Institute of Science.
The second author is supported by  J. C. Bose Fellowship from the Department of Science and Technology, Government of India.

\end{document}